\documentclass[11pt]{article} 

\makeatletter
\def\maketitle{
	\begin{center}
	\begin{center}
		{\LARGE \@title}
		\vskip 0.5cm
		{\Large \@author}
		\vskip 0.5cm
		{\@affiliation}
		\vskip 0.5cm
		{\@email}
		\vskip 0.5cm
		{\Large \@date}
		\vspace*{1.5cm}
	\end{center}
	\end{center}
}
\def\email#1{\def\@email{#1}}
\def\affiliation#1{\def\@affiliation{#1}}
\makeatother

\usepackage{amsmath}
\usepackage{amssymb}
\usepackage[T1]{fontenc}
\usepackage[latin1]{inputenc}
\def\r{\rightarrow}

\usepackage{enumerate}

\newcommand{\fdem}{\hspace*{\fill}~$\Box$\par\endtrivlist\unskip}

\newcommand{\E}{\mathbb{E}}     
     
\renewcommand{\L}{\mathbb{L}}

\newcommand{\N}{\mathbb{N}}     
\newcommand{\Z}{\mathbb{Z}}
\newcommand{\R}{\mathbb{R}}     
     
\newcommand{\C}{\mathbb{C}} 
 
\newcommand{\X}{\mathbb{X}}

\renewcommand{\dim}{\mathop{\rm dim}}
\renewcommand{\ker}{\mathop{\rm Ker}}

\renewcommand{\r}{\mathop{\rightarrow}}

\newcommand{\cB}{\mbox{$\cal B$}}

\newcommand{\cE}{\mbox{$\cal E$}}

\newcommand{\cM}{\mbox{$\cal M$}}

\newcommand{\cR}{\mbox{$\cal R$}}

\newcommand{\cX}{\mbox{$\cal X$}}

\newcommand{\cV}{\mbox{$\cal V$}}

\newcommand{\gc}{\widehat\gamma}
\newcommand{\dc}{\widehat\delta}
\newcommand{\rc}{\widehat{r}_{ess}}
\newcommand{\rhoc}{\widehat\rho}

\newenvironment{proof}[1]{\textit{Proof#1.\,}}{\fdem}
\newtheorem{rem}{Remark}[section]
\newtheorem{ex}{Example}
\newtheorem{cor}{Corollary}[section]
\newtheorem{lem}{Lemma}[section]
\newtheorem{pro}[lem]{Proposition}

\title{Spectral analysis of Markov kernels  and application to the convergence rate of discrete random walks}

\author{Lo\"ic \textsc{Herv\'e} and James \textsc{Ledoux} \footnote{postal address : INSA de Rennes, 20 avenue des Buttes de Coesmes, CS 70 839
35708 Rennes Cedex 7  }
}
\affiliation{INSA de Rennes, F-35708, France; IRMAR CNRS-UMR 6625, F-35000, France; Universit\'e Europ\'eenne de Bretagne, France.}

\email{\{Loic.Herve,James.Ledoux\}@insa-rennes.fr}
\begin{document}

\maketitle
\begin{abstract}
Let $\{X_n\}_{n\in\N}$ be a Markov chain on a measurable space $\X$ with transition kernel $P$  and let $V:\X\r[1,+\infty)$. The Markov kernel $P$ is here considered as a linear bounded operator on the weighted-supremum space $\cB_V$ associated with $V$. Then the combination of quasi-compactness arguments with precise analysis of eigen-elements of $P$ allows us to estimate the geometric rate of convergence $\rho_V(P)$ of $\{X_n\}_{n\in\N}$ to its invariant probability measure in operator norm on $\cB_V$. A general procedure to compute $\rho_V(P)$ for discrete Markov random walks with identically distributed  bounded increments is specified.

\end{abstract}
\begin{center}

\end{center}

\section{Introduction} \label{intro}

Let $(\X,\cX)$ be a measurable space with a $\sigma$-field $\cX$, and let $\{X_n\}_{n\ge 0}$ be a Markov chain
with state space $\X$ and transition kernels $\{P(x,\cdot) : x\in \X\}$.  Let $V:\X\r[1,+\infty)$. Assume that $\{X_n\}_{n\ge 0}$ has an invariant probability measure $\pi$ such that $\pi(V) :=\int_{\X} V(x) \pi(dx) <\infty$. This paper is based on  the connection between spectral properties of the Markov kernel $P$ and the so-called $V$-geometric ergodicity \cite{MeyTwe93} which is  the following convergence property for some constants $c_{\rho}>0$ and $\rho\in(0,1)$: 
\begin{equation} \label{first-V-geo}
\sup_{|f|\le V} \sup_{x\in\X} \frac{\big|\E[f(X_n)\mid X_0=x] - \pi(f)\big|}{V(x)} \le c_{\rho} \,\rho^n.
\end{equation}
Let us introduce the weighted-supremum Banach space $(\cB_V,\|\cdot\|_V)$ composed of  measurable functions $f : \X\r\C$ such that 
$$\|f\|_V  := \sup_{x\in \X} \frac{|f(x)|}{V(x)} < \infty.$$
Then (\ref{first-V-geo}) reads as $\|P^nf- \pi(f) 1_{\X}\|_V \le c_{\rho} \rho^n$ for any $f\in\cB_V$ such that $\|f\|_V\leq1$, 
and there is a great interest in obtaining upper bounds for \textit{the convergence rate $\rho_V(P)$} defined by 
\begin{equation} \label{Def_rhoV}
	\rho_V(P) := \inf\big\{ \rho \in (0,1),\sup_{\|f\|_V\leq1}\|P^nf-\pi(f) 1_{\X}\|_V = O(\rho^n)\big\}.
\end{equation}
For irreducible and aperiodic discrete Markov chains, criteria for the $V$-geometric ergodicity are well-known from the literature using either the equivalence between geometric ergodicity and $V$-geometric ergodicity of $\N$-valued Markov chains \cite[Prop.~2.4]{HorSpi92}, or the strong drift condition. For instance, when $\X:=\N$ (with $\lim_{n}V(n)= +\infty$), the strong drift condition is 
	\[ P V \le \varrho V + b \, 1_{\{0,1,\ldots,n_0\}}
\]
for some $\varrho <1, b <\infty$ and $n_0\in\N$ (see \cite{MeyTwe93}). Estimating $\rho_V(P)$ from the parameters $\varrho, b, n_0$ is a difficult issue. This often leads to  unsatisfactory bounds, except for stochastically monotone~$P$ (see \cite{MeyTwe94,LunTwe96,Bax05} and the references therein).

This work presents a new procedure to study the convergence rate $\rho_V(P)$ under the following weak drift condition 
\begin{equation} \label{cond-D}
\exists N\in\N^*,\ \exists d \in(0,+\infty),\ \exists \delta\in(0,1),\quad P^NV \leq \delta^N\, V + d\, 1_{\X}. \tag{\textbf{WD}}
\end{equation} 
The $V$-geometric ergodicity clearly implies (\ref{cond-D}). Conversely, such a condition with $N=1$ was introduced in \cite[Lem.~15.2.8]{MeyTwe93} as an alternative to the drift condition \cite[(V4)]{MeyTwe93} to obtain the $V$-geometric ergodicity under suitable assumption on $V$. Note that, under Condition~(\ref{cond-D}), the following real number $\delta_V(P)$ is well defined:
\begin{equation*} 
\delta_V(P) := \inf\big\{\delta\in[0,1) : \exists N\in\N^*,\, \exists d\in(0,+\infty),\ P^NV \leq \delta^N\, V + d\, 1_{\X}  \big\}.
\end{equation*}

A spectral analysis of $P$ is presented in Section~\ref{sec-mino} using quasi-compactness. More specifically,  when the Markov kernel $P$ has an invariant probability distribution, the connection between the $V$-geometric ergodicity and the quasi-compactness of $P$ is made explicit in Proposition~\ref{CNS-qc-Vgeo}. Namely, $P$ is $V$-geometrically ergodic if and only if $P$ is a power-bounded quasi-compact operator on $\cB_V$ for which $\lambda=1$ is a simple eigenvalue and the unique eigenvalue of modulus one. In this case, if  $r_{ess}(P)$ denotes the essential spectral radius  of $P$ on $\cB_V$ (see (\ref{def-q-c})) and if $\cV$ denotes the set of eigenvalues $\lambda$ of $P$ such that $r_{ess}(P)<|\lambda|<1$, then the convergence rate $\rho_V(P)$ is given by (Proposition~\ref{CNS-qc-Vgeo}): 
\begin{equation} \label{intro-gap-r-ess}
\rho_V(P)=r_{ess}(P)\ \text{ if }\ \cV=\emptyset \quad \text{ and } \quad \rho_V(P)=\max\{|\lambda|,\, \lambda\in\cV\}\ \text{ if }\ \cV\neq\emptyset.
\end{equation} 
Interesting bounds for generalized eigenfunctions $f\in\cB_V\cap\ker(P-\lambda I)^p$ associated with $\lambda\in\cV$ are presented in Proposition~\ref{pro-tail-fct-propre}.  
Property~(\ref{intro-gap-r-ess}) is relevant to study the convergence rate $\rho_V(P)$ provided that, first an accurate bound of $r_{ess}(P)$ is known, second the above set $\cV$ is available. Bounds of $r_{ess}(P)$ related to drift conditions can be found in \cite{Wu04} and \cite{HerLed14} under various assumptions (see Subsection~\ref{sec-qc-V-geo}). In view of our applications, let us just mention that $r_{ess}(P)= \delta_V(P)$ in case $\X:=\N$ and $\lim_n V(n) = +\infty$ (see Proposition~\ref{cor-qc-bis}). However, even if the state space is discrete, finding the above set $\cV$ is difficult. 

In Section~\ref{sub-sec-countable}, the above spectral analysis is applied to compute the rate of convergence $\rho_V(P)$ of discrete Random Walks (RW). In particular, a complete solution is presented for RWs with identically distributed (i.d.) bounded increments.  In fact, Proposition~\ref{pro-rate-HRW} allows us to formulate an algebraic  procedure based on polynomial eliminations providing $\rho_V(P)$ (see  Corollary~\ref{cor-step-3}). To the best of our knowledge, this general result is new. Note that it requires neither reversibility  nor stochastic monotonicity of $P$.  

This procedure is illustrated in Section~\ref{rate-general-procedure}. First we consider the case of birth-and-death Markov kernel $P$ defined by 
$P(0,0):=a$ and $P(0,1):=1-a$ for some $a\in(0,1)$ and by 
\begin{equation*} \label{intro-hyp-ref-rand}
 \forall n\ge 1,\ P(n,n-1) :=p,\quad P(n,n) := r,\quad P(n,n+1) := q, 
\end{equation*} 
where $p,q,r\in[0,1]$ are such that $p+r+q=1$, $p>q>0$. 
Explicit formula for $\rho_V(P)$ with respect to $V:=\{(p/q)^{n/2}\}_{n\in\N}$ is given in Proposition~\ref{BD}. When $r:=0$, such a result has been obtained for $a< p$ in \cite{RobTwe99} and \cite[Ex.~8.4]{Bax05} using Kendall's theorem, and for $a\geq p$ in \cite{LunTwe96} using the stochastic monotony of $P$. Our method gives a unified and simpler computation of  $\rho_V(P)$ which moreover encompasses the case $r\neq0$. For general RWs with i.d.~bounded increments, the elimination procedure requires to use symbolic computations. The second example illustrates this point with the non reversible RW defined by 
\begin{equation*} 
\forall n\geq 2,\ P(n,n-2) = a_{-2},\ P(n,n-1) = a_{-1},\ P(n,n) = a_{0}, \ P(n,n+1) = a_1
\end{equation*}
for any nonnegative $a_i$ satisfying $a_{-2}+ a_{-1} + a_{0} + a_1 =1$, $a_{-2}>0$, $2a_{-2} + a_{-1} > a_1>0$, and for any finitely many boundary transition probabilities.
In Section~\ref{sect-unbounded}, specific examples of RWs on $\X:=\N$ with unbounded increments considered in the literature are investigated.

To conclude this introduction, we mention a point which can be source of confusion in a first reading. In this paper, we are concerned with the convergence rate (\ref{Def_rhoV}) with respect to some weighted-supremun Banach space $\cB_V$. Thus, we do not consider here the decay parameter or the convergence rate of ergodic Markov chains in the usual Hilbert space $\L^2(\pi)$  which is related to spectral properties of the transition kernel with respect to this space. In particular,  for Birth-and-Death Markov chains, we can not compare our results with those of \cite{DooSch95} on the $\ell^2(\pi)$-spectral gap  and the decay parameter. A detailed discussion is provided in Remark~\ref{l2-spectral-gap}. 

\section{Quasi-compactness on $\cB_V$ and $V$-geometric ergodicity} \label{sec-mino}

We assume that $P$ satisfies (\ref{cond-D}). Then $P$ continuously acts on $\cB_V$, and iterating (\ref{cond-D}) shows that $P$ is power-bounded on $\cB_V$, namely $\sup_{n\geq1}\|P^n\|_V<\infty$,  where $\|\cdot\|_V$ also stands for the operator norm on $\cB_V$. Thus we have $r(P):=\lim_n\|P^n\|_V^{1/n}=1$ since $P$ is Markov. 
\subsection{From quasi-compactness on $\cB_V$ to $V$-geometric ergodicity} \label{sec-qc-V-geo}

Let $I$ denote the identity operator on $\cB_V$. Recall that $P$ is said to be quasi-compact on $\cB_V$ if there exist $r_0\in(0,1)$ and $m\in\N^*$, $\lambda_i\in\C$, $p_i\in\N^*$ ($\, i=1,\ldots,m$) such that: 
\begin{subequations}
\begin{equation} \label{noyit}
\cB_V = \overset{m}{\underset{i=1}{\oplus}} \ker(P - \lambda_i I)^{p_i}\, \oplus H, 
\end{equation}
where the $\lambda_i$'s are such that 
\begin{equation} \label{noyit-lambda}
|\lambda_i| \geq r_0\quad \text{ and } \quad 1\le \dim\ker(P-\lambda_i I)^{p_i} < \infty,
\end{equation}
and $H$ is a closed $P$-invariant subspace such that 
\begin{equation} \label{noyit-H}
\inf_{n\geq1}\big(\sup_{h\in H,\, \|h\|\leq1}\|P^nh\|\big)^{1/n} < r_0.\end{equation}
\end{subequations}
Concerning the essential spectral radius of $P$, denoted by $r_{ess}(P)$, here it is enough to have in mind that, if $P$ is quasi-compact on $\cB_V$, then we have (see for instance \cite{Hen93})
\begin{equation} \label{def-q-c} 
r_{ess}(P) := \inf\big\{r_0\in(0,1) \text{ such that~ (\ref{noyit})-(\ref{noyit-H}) hold}  \big\}.
\end{equation}

As mentioned in Introduction, the essential spectral radius of Markov kernels acting on $\cB_V$ is studied in \cite{Wu04,HerLed14}. For instance, under Condition (\ref{cond-D}), the following result is proved in \cite{HerLed14}: if $P^\ell$ is compact from $\cB_0$ to $\cB_V$ for some $\ell\geq 1$, where $(\cB_0,\|\cdot\|_0)$ is the  Banach space composed of bounded measurable functions $f : \X\r\C$ equipped with the supremum norm $\|f\|_0:=\sup_{x\in\X}|f(x)|$, then $P$ is quasi-compact on $\cB_V$ with 
$$r_{ess}(P) \leq \delta_V(P).$$  
Moreover, equality $r_{ess}(P)= \delta_V(P)$ holds in many situations, in particular in the discrete state case with $V(n)\r \infty$ (see Proposition~\ref{cor-qc-bis}). 

Next we explicit a result which makes explicit the relationship between the quasi-compactness of $P$ and the $V$-geometric ergodicity of the Markov chain $\{X_n\}_{n \in\N}$ with transition kernel $P$. Moreover, we provide an explicit formula for $\rho_V(P)$ in terms of the spectral elements of $P$. Note that for any ${r_0}\in (r_{ess}(P),1)$, the set of all the eigenvalues of $\lambda$ of $P$ such that ${r_0}\leq|\lambda|\leq1$ is finite (use (\ref{def-q-c})). 
\begin{pro} \label{CNS-qc-Vgeo}
Let $P$ be a transition kernel which has an invariant probability measure $\pi$ such that $\pi(V)<\infty$. The two following assertions are equivalent:
\begin{enumerate}[(a)]
\item $P$ is $V$-geometrically ergodic.
\item $P$ is a power-bounded quasi-compact operator on $\cB_V$, for which $\lambda=1$ is a simple eigenvalue  (i.e.~$\ker(P-I) = \C\cdot 1_{\X}$) and the unique eigenvalue of modulus one.
\end{enumerate}
Under any of these conditions, we have $\rho_V(P) \ge r_{ess}(P)$. In fact, for ${r_0}\in(r_{ess}(P),1)$, denoting the set of all the eigenvalues $\lambda$ of $P$ such that ${r_0}\leq|\lambda|<1$ by $\cV_{r_0}$, we have:
\begin{itemize}
    \item either $\rho_V(P) \leq {r_0}$ when $\cV_{r_0}=\emptyset$,
    \item or $\rho_V(P) = \max\{|\lambda|,\, \lambda\in\cV_{r_0}\}$ when $\cV_{r_0}\neq\emptyset$.
\end{itemize}
Moreover, if  $\cV_{r_0}=\emptyset$ for all ${r_0}\in(r_{ess}(P),1)$, then $\rho_V(P) =r_{ess}(P)$.
\end{pro}

The $V$-geometric ergodicity of $P$ obviously implies that $P$ is quasi-compact on $\cB_V$ with $\rho_V(P) \geq r_{ess}(P)$ (see e.g.~\cite{KonMey03}). This follows from (\ref{def-q-c}) using $H:=\{f\in\cB_V : \pi(f)=0\}$ in (\ref{noyit})-(\ref{noyit-H}). The property that $P$ has a spectral gap on $\cB_V$ in the recent paper \cite{KonMey12} corresponds here to the quasi-compactness of $P$ (which is a classical terminology in spectral theory). The spectral gap in \cite{KonMey12} corresponds to the value $1-\rho_V(P)$. Then, \cite[Prop.~1.1]{KonMey12}) is another formulation, under $\psi$-irreducibility and aperiodicity assumptions,  of the equivalence of properties (a) and (b) in Proposition~\ref{CNS-qc-Vgeo} (see also \cite[Lem.~2.1]{KonMey12}). Details on the proof of Proposition~\ref{CNS-qc-Vgeo} are provided in \cite{GuiHerLed11}. 
For general quasi-compact Markov kernels on $\cB_V$, the result \cite[Th.~4.6]{Wu04} also provides interesting additional material on peripheral eigen-elements. 
The next subsection completes the previous  spectral description by providing bounds for the  generalized eigenfunctions associated with eigenvalues $\lambda$  such that $\delta \leq |\lambda| \leq 1$, with $\delta$ given in (\ref{cond-D}).
%
\subsection{Bound on generalized eigenfunctions of $P$} \label{subsec-taille-fct-propre}
%
\begin{pro} \label{pro-tail-fct-propre}
Assume that the weak drift condition \emph{(\ref{cond-D})} holds true.
If $\lambda\in\C$ is such that $\delta \leq |\lambda| \leq 1$, with $\delta$ given in \emph{(\ref{cond-D})}, and if $f\in\cB_V\cap\ker(P-\lambda I)^p$ for some $p\in\N^*$, then there exists $c\in(0,+\infty)$ such that 
$$|f| \leq c\, V^{\frac{\ln|\lambda|}{\ln\delta}}\, (1+\ln V)^{\frac{p(p-1)}{2}}.$$  
\end{pro}
Thus, if $\lambda$ is an eigenvalue such that $|\lambda|=1$, then any associated eigenfunction $f$ is bounded on $\X$. By contrast, if $|\lambda|$ is close to $\delta_V(P)$, then $|f| \leq c\, V^{\beta(\lambda)}$ with $\beta(\lambda)$ close to~1. The proof of Proposition~\ref{pro-tail-fct-propre} is based on the following lemma.
\begin{lem} \label{lem-ineq-lem-tail-fct}
Let $\lambda\in\C$ be such that $\delta \leq |\lambda| \leq 1$. Then 
\begin{equation} \label{ineq-lem-tail-fct}
\forall f\in\cB_V,\ \exists c\in(0,+\infty),\ \forall x\in\X,\quad |\lambda|^{-n(x)} \big|(P^{n(x)}f)(x)\big|  \leq c\, V(x)^{\frac{\ln|\lambda|}{\ln\delta}}
\end{equation}
with, for any $x\in\X$, $n(x) := \big\lfloor \frac{-\ln V(x) }{\ln\delta}\big\rfloor$ where $\lfloor \cdot\rfloor$ denotes the integer part function. 
\end{lem}
\begin{proof}{}
First note that the iteration of (\ref{cond-D}) gives 
\begin{equation*} \label{itération-cond-D}
\forall k\geq 1,\quad P^{kN}V \leq \delta^{kN}\, V + d\big(\sum_{j=0}^{k-1}\delta^{jN}\big)\, 1_{\X} \leq \delta^{kN}\, V +  
\frac{d}{1-\delta^N}\, 1_{\X}.
\end{equation*}
Let $g\in\cB_V$ and $x\in\X$. Using the last inequality, the positivity of $P$ and $|g|\leq \|g\|_V\, V$, we obtain with $b := d/(1-\delta^N)$: 
\begin{equation} \label{val-pro}
\forall k\geq1,\quad |(P^{kN}g)(x)| \leq (P^{kN}|g|)(x) \leq \|g\|_V\, (P^{kN}V)(x) \leq  \|g\|_V\big(\delta^{kN}V(x) + b\big).
\end{equation}
The previous inequality is also fulfilled with $k=0$. Next, let $f\in\cB_V$ and $n\in\N$. Writing $n = kN+r$, with $k\in\N$ and $r\in\{0,1,\ldots,N-1\}$, and applying (\ref{val-pro}) to $g:=P^rf$, we obtain with $\xi:=\max_{0\leq\ell\leq N-1}\|P^\ell f\|_V$ (use $P^nf = P^{kN}(P^rf)$): 
\begin{equation} \label{val-pro-bis}
\big|(P^nf)(x)\big| \leq \xi\big[\delta^{kN}V(x) + b\big] \leq 
\xi\big[\delta^{-r}\big( \delta^{n}V(x) + b\big)\big] \leq \xi\,\delta^{-N}\big(\delta^{n}V(x) + b\big).
\end{equation}
Using the inequality $$-\frac{\ln V(x)}{\ln\delta} -1 \leq n(x) \leq -\frac{\ln V(x)}{\ln\delta}$$ 
and the fact that  $\ln\delta \leq \ln|\lambda|\leq 0$, Inequality (\ref{val-pro-bis}) with $n:=n(x)$ gives: 
\begin{eqnarray*}
|\lambda|^{-n(x)} \big|(P^{n(x)}f)(x)\big| &\leq& \xi\, \delta^{-N}\bigg(\big(\delta|\lambda|^{-1}\big)^{n(x)}\, V(x) + b\, |\lambda|^{-n(x)} \bigg) \\
& &= \xi\, \delta^{-N}\bigg(e^{n(x)(\ln\delta - \ln|\lambda|)}\, e^{\ln V(x)} + b\, e^{-n(x)\ln|\lambda|} \bigg) \\
&\leq&  \xi\, \delta^{-N}\bigg(e^{(\frac{\ln V(x) }{\ln\delta}+1)\, (\ln|\lambda|-\ln\delta)}\, e^{\ln V(x)} + 
b\, e^{\frac{\ln V(x)}{\ln\delta}\ln|\lambda|} \bigg)  \\
&& = \xi\, \delta^{-N}\bigg(e^{\frac{\ln|\lambda|}{\ln\delta}\ln V(x)}\, e^{\ln|\lambda|-\ln\delta}\,  + b\, 
V(x)^{\frac{\ln|\lambda|}{\ln\delta}}\bigg) \\
&& =\xi\, \delta^{-N} \big(e^{\ln|\lambda|-\ln\delta}+b\big)\, V(x)^{\frac{\ln|\lambda|}{\ln\delta}}.
\end{eqnarray*}
This gives Inequality~(\ref{ineq-lem-tail-fct}) with $c:=\xi\, \delta^{-N} (e^{\ln|\lambda|-\ln\delta}+b)$. 
\end{proof}
\begin{proof}{ of Proposition~\ref{pro-tail-fct-propre}}
 If $f\in\cB_V\cap\ker(P-\lambda I)$, then $|\lambda|^{-n(x)} |(P^{n(x)}f)(x)| = |f(x)|$, so that (\ref{ineq-lem-tail-fct}) gives the expected conclusion when $p=1$. Next, let us  proceed by induction. Assume that the conclusion of Proposition~\ref{pro-tail-fct-propre} holds for some $p\geq 1$. Let $f\in\cB_V\cap\ker(P-\lambda I)^{p+1}$. We can write 
\begin{equation} \label{inter-tail-1}
P^nf = (P-\lambda I+\lambda I)^n f = \lambda^{n}\, f + \sum_{k=1}^{\min(n,p)} \binom{n}{k} \lambda^{n-k}\,  (P-\lambda I)^k f.
\end{equation}
For $k\in\{1,\ldots,p\}$, we have $f_k := (P-\lambda I)^k f \in\ker(P-\lambda I)^{p+1-k} \subset \ker(P-\lambda I)^{p}$, thus we have from the induction hypothesis : 
\begin{equation} \label{inter-tail-2}
\exists c'\in(0,+\infty),\ \forall k\in\{1,\ldots,p\}, \ \forall x\in\X, \quad |f_k(x)| \leq c'\, V(x)^{\frac{\ln|\lambda|}{\ln\delta}}\, (1+\ln V(x))^{\frac{p(p-1)}{2}}.
\end{equation}
Now, we obtain from (\ref{inter-tail-1}) (with $n:=n(x)$), (\ref{inter-tail-2}) and Lemma~\ref{lem-ineq-lem-tail-fct} that for all $x\in\X$: 
\begin{eqnarray*}
|f(x)| &\leq& |\lambda|^{-n(x)} \big|(P^{n(x)}f)(x)\big| + c'\, V(x)^{\frac{\ln|\lambda|}{\ln\delta}}\, (1+\ln V(x))^{\frac{p(p-1)}{2}}\, |\lambda|^{-\min(n,p)} \sum_{k=1}^{\min(n,p)} \binom{n(x)}{k} \\
&\leq&  c\, V(x)^{\frac{\ln|\lambda|}{\ln\delta}} + c_1\, V(x)^{\frac{\ln|\lambda|}{\ln\delta}}\, (1+\ln V(x))^{\frac{p(p-1)}{2}}\, n(x)^p \\
&\leq& c_2 V(x)^{\frac{\ln|\lambda|}{\ln\delta}}\, (1+\ln V(x))^{\frac{p(p-1)}{2} + p}
\end{eqnarray*}
with some constants $c_1,c_2\in(0,+\infty)$ independent of $x$. 
This gives the expected result.
\end{proof}
%

\section{Spectral properties of discrete Random Walks} \label{sub-sec-countable}
In the sequel, the state space $\X$ is discrete. For the sake of simplicity, we assume that $\X:=\N$. Let $P=(P(i,j))_{i,j\in\N^2}$ be a Markov kernel on $\N$. The function $V:\N\r[1,+\infty)$ is assumed to satisfy 
$$\lim_nV(n) = +\infty \quad \text{ and } \quad \sup_{n\in\N}\frac{(PV)(n)}{V(n)}<\infty.$$ 

The first focus is on the estimation of $r_{ess}(P)$ from Condition~(\ref{cond-D}). 
\begin{pro} \label{cor-qc-bis} 
Let $\X:=\N$. 
The two following conditions are equivalent: 
\begin{enumerate}[(a)]
	\item Condition~\emph{(\ref{cond-D})} holds with $V$; 
	\item $\displaystyle L := \inf_{N\geq 1}(\ell_N)^{\frac{1}{N}}< 1$ where 
	$\ell_N := \displaystyle \limsup_{n\r+\infty}\frac{(P^NV)(n)}{V(n)}$.
\end{enumerate}
In this case, $P$ is power-bounded and quasi-compact on $\cB_V$ with $r_{ess}(P) = \delta_V(P) = L$.
\end{pro}
The proof of the equivalence $(a)\Leftrightarrow (b)$, as well as the equality $\delta_V(P) = L$, is straightforward (see \cite[Cor.~4]{GuiHerLed11}). That $P$ is quasi-compact on $\cB_V$ under (\ref{cond-D}) in the discrete case, with $r_{ess}(P) \leq \delta_V(P)$, can be derived from \cite{Wu04} or \cite{HerLed14} (see Subsection~\ref{sec-qc-V-geo} and use the fact that the injection from $\cB_0$ to $\cB_V$ is compact when $\X:=\N$ and $\lim_n V(n)=+\infty$). Equality $r_{ess}(P) = \delta_V(P)$ can be proved by combining the results \cite{Wu04,HerLed14} (see \cite[Cor.~1]{GuiHerLed11} for details). 

In Sections~\ref{sub-sec-countable} and \ref{rate-general-procedure}, sequences of the special form $ V_\gamma:=\{\gamma^n\}_{n\in\N}$ for some $\gamma\in(1,+\infty)$ will be considered. The associated weighted-supremum  space $\cB_\gamma \equiv \cB_{V_\gamma}$ is defined by: 
\begin{equation*} \label{def-Bgamma-discret}
\cB_{\gamma} := \big\{\{f(n)\}_{n\in\N}\in\C^{\N} : \sup_{n\in\N}\gamma^{-n}|f(n)| < \infty \big\}.
\end{equation*}

\subsection{Quasi-compactness of RWs with bounded state-dependent increments} \label{ex-reflecting-nonhom}  
Let us fix $c,g,d\in\N^*$, and assume that the kernel $P$ satisfies the following conditions: 
\begin{subequations}  
\begin{gather}
\forall i\in\{0,\ldots,g-1\},\quad \sum_{j=0}^c P(i,j)=1; \label{Def_Bord_NHRW} \\
\forall i\ge g, \forall j\in\N, \quad P(i,j) = 
\begin{cases}
 a_{j-i}(i) & \text{ if }\  i-g\leq j \leq i+d \\
 0 & \text{ otherwise}  \\
\end{cases} \label{Def_NHRW}
\end{gather} 
\end{subequations}
where $(a_{-g}(i),\ldots,a_d(i))\in[0,1]^{g+d+1}$ satisfies $\sum_{k=-g}^{d} a_k(i)=1$ for all $i\ge g$. This kind of kernels arises, for instance, from time-discretization of Markovian queuing models. Note that more general models and their use in queuing theory are discussed in \cite{KliDud06}. In particular, conditions for
(non) positive recurrence are provided. 

\begin{pro} \label{pro-non-hom}
Assume that, for every $k\in\Z$ such that $-g\leq k \leq d$, $\lim_n a_k(n)=a_k\in[0,1]$,
and that   
\begin{gather} 
 \exists \gamma\in(1,+\infty): \qquad \phi(\gamma) := \sum_{k=-g}^{d} a_{k}\, \gamma^{k} <1. \label{non-hom-cont}
\end{gather}
Then $P$ satisfies Condition~\emph{(\ref{cond-D})} with $\delta=\phi(\gamma)$. Moreover $P$ is power-bounded and quasi-compact on $\cB_{\gamma}$ with $r_{ess}(P) = L = \phi(\gamma)$.
\end{pro}

\begin{lem} \label{NERI} When $a_{-g}$ and $a_d$ are positive, Condition~\emph{(\ref{non-hom-cont})} is equivalent to 
\begin{equation} \label{MoySautBorne}
\sum_{k=-g}^{d} k\, a_{k}\, < 0. \tag{\textbf{NERI}}
\end{equation}
Then, there exists a unique real number $\gamma_0 >1$ such that $\phi(\gamma_0)=1$ and  
\[\forall \gamma \in(1,\gamma_0), \quad \phi(\gamma)<1 \]
and there is a unique $\widehat{\gamma}$ such that 
\begin{equation*} \label{def-delta}
 \dc:=\phi(\widehat\gamma) = \min_{\gamma\in(1,\infty)} \phi(\gamma) = \min_{\gamma\in(1,\gamma_0)} \phi(\gamma) <1.	
\end{equation*}
\end{lem}
Condition~(\ref{MoySautBorne}) means that the expectation of the probability distribution of the random increment is negative. 
Although the results of the paper on RWs with i.d.~bounded increments involving Condition~(\ref{MoySautBorne}) and $a_{-g},a_d>0$ will be valid for $\gamma \in(1,\gamma_0)$, only this value $\widehat\gamma$ is considered in the statements. Note that the essential spectral radius $r_{ess}(P_{|\cB_{\gc}})$ of $P$ with respect to $\cB_{\gc}$, which will be denoted by $\rc(P)$ in the sequel, is the smallest value of $r_{ess}(P_{|\cB_{\gamma}})$ on $\cB_{\gamma}$  for $\gamma\in(1,\gamma_0)$. When $\gamma \nearrow \gamma_0$, the essential spectral radius $r_{ess}(P_{|\cB_{\gamma}}) \nearrow 1$ since the space $\cB_\gamma$ becomes large.  
When $\gamma \searrow 1$, then $r_{ess}(P_{|\cB_{\gamma}}) \nearrow 1$ since $\cB_\gamma$ becomes close to the space $\cB_0$ of bounded functions. In this case, the geometric ergodicity is lost since the RWs are typically not uniformly ergodic (i.e. $V\equiv 1$) due the non quasi-compactness of $P$ on  $\cB_0$.  
\begin{ex}[State-dependent birth-and-death Markov chains] \label{ex-bin-non-hom} 
When $c=g=d:=1$ in \emph{(\ref{Def_Bord_NHRW})}-\emph{(\ref{Def_NHRW})}, we obtain the standard class of state-dependent birth-and-death Markov chains:
\begin{gather*}
P(0,0) := r_0,\quad P(0,1) := q_0\\
\forall n\ge 1,\ P(n,n-1) := p_n,\quad P(n,n) := r_n,\quad P(n,n+1) := q_n,
\end{gather*}
where 
$(p_0,q_0)\in[0,1]^2, p_0+q_0=1$ and $(p_n,r_n,q_n)\in[0,1]^3, p_n+r_n+q_n=1$. Assume that:  
$$\lim_n p_n := p  \quad \lim_n r_n := r ,\quad \lim_n q_n :=q.$$
If $\gamma\in(1,+\infty)$ is such that $\phi(\gamma) := p/\gamma + r + q\gamma <1$ 
then it follows from Proposition~\ref{pro-non-hom} that 
$r_{ess}(P) = p/\gamma + r + q\gamma$. The conditions $\gamma >1$ and $p/\gamma + r + q\gamma <1$ are equivalent to the following ones (use $r=1-p-q$ for $(i)$): 
\begin{center}
\begin{tabular}{rl}
  $(i)$ & either $q>0$, $q-p<0$ (i.e. \emph{(\ref{MoySautBorne})}) 
  and $1<\gamma < \gamma_0=p/q$;\\
   $(ii)$ & or $q=0$, $p>0$ and $\gamma>1$. 
\end{tabular}
\end{center}
\leftmargini 1.5em
\begin{itemize}
	\item[(i)] When $p>q>0$ and $1<\gamma < \gamma_0$: 
	$P$ is power-bounded and quasi-compact on $\cB_{\gamma}$ with $r_{ess}(P) = \phi(\gamma)$. Set $\widehat\gamma := \sqrt{\gamma_0}=\sqrt{p/q}\in(1,\gamma_0)$. Then 
$\min_{\gamma>1} \phi(\gamma) = \phi(\widehat\gamma) = r+2\sqrt{pq}$ 
	and  the essential spectral radius $\hat{r}_{ess}(P)$ on $\cB_{\widehat\gamma}$ satisfies $\widehat{r}_{ess}(P) = r+2\sqrt{pq}$.
  \item[(ii)] When $q:=0, p>0$ and $\gamma>1$: 
  $r_{ess}(P) = \phi(\gamma)=p/\gamma + r$.
\end{itemize}
\end{ex} 
\begin{rem}
If $c$ is allowed to be $+\infty$ in Condition~\emph{(\ref{Def_Bord_NHRW})}, that is  
\begin{gather} 
  \forall i\in\{0,\ldots,g-1\},\quad \sum_{j\ge 0} P(i,j)\gamma^j < \infty, \label{non-hom-moment} 
\end{gather}
then the conclusions of Proposition~\ref{pro-non-hom} and Example~\ref{ex-bin-non-hom} are still valid under the additional Condition (\ref{non-hom-moment}).
\end{rem}
\begin{proof}{ of Proposition~\ref{pro-non-hom}} 
Set $\phi_n(\gamma) := \sum_{k=-g}^{d} a_{k}(n)\, \gamma^{k}$. We have $(PV_\gamma)(n) = \phi_n(\gamma) V_\gamma(n)$ for each $n\geq g$. Thus $\ell_1 = \lim_n \phi_n(\gamma) =  \phi(\gamma)$. Now assume that $\ell_{N-1} := \lim_n (P^{N-1}V)(n)/V(n) = \phi(\gamma)^{N-1}$ for some $N\geq1$. Since 
$$\forall i\geq Ng,\quad (P^NV)(i) = \sum_{j=-g}^{d} a_j(i)\, (P^{N-1}V)(i+j)$$
we obtain 
$$\frac{(P^NV)(i)}{V(i)} = \sum_{j=-g}^{d} a_j(i) \gamma^j\,  \frac{(P^{N-1}V)(i+j)}{\gamma^{i+j}} \xrightarrow[i\r+\infty]{} \phi(\gamma)\, \phi(\gamma)^{N-1}.$$
Hence $\ell_N = \phi(\gamma)^{N}$, and $\phi(\gamma)=L=r_{ess}(P)$ from Proposition~\ref{cor-qc-bis}. 
\end{proof}
\noindent\begin{proof}{ of Lemma~\ref{NERI}}Since the second derivative of $\phi$ is positive on $(0,+\infty)$, $\phi$ is convex on $(0,+\infty)$. When $a_{-g}$ and $a_d$ are positive then $\lim_{t\r 0^{+}} \phi(t)= \lim_{t\r+\infty} \phi(t)= +\infty$ and, 
 since $\phi(1)=1$,  Condition~(\ref{non-hom-cont}) is  equivalent to $\phi'(1) < 0$, that is (\ref{MoySautBorne}). The other properties of $\phi(\cdot)$ are immediate.
\end{proof}
\subsection{Spectral analysis of RW~with i.d.~bounded increments} \label{sect-rate-general}

Let $P:=(P(i,j))_{(i,j)\in\N^2}$ be the transition kernel of a RW~with i.d.~bounded increments. Specifically we assume that there exist some positive integers $c,g,d\in\N^*$ such that 
\begin{subequations}
\begin{gather}
\forall i\in\{0,\ldots,g-1\},\quad \sum_{j = 0}^{c} P(i,j)=1; \label{cond-bound-prob} \\
\forall i\ge g, \forall j\in\N, \quad P(i,j) = \begin{cases}
 a_{j-i} & \text{if }\  i-g\leq j \leq i+d \\
 0 & \text{otherwise.}  \\
\end{cases} \label{Def_HRW-ter} \\
(a_{-g},\ldots,a_d)\in[0,1]^{g+d+1} : a_{-g}>0, \ a_d>0, \ \sum_{k=-g}^{d} a_k=1. 
\label{Def_HRW-bis} 
\end{gather} 
\end{subequations}
Let us assume that Condition~(\ref{MoySautBorne}) holds. We know from Lemma~\ref{NERI} and Proposition~\ref{pro-non-hom} that $P$ is quasi-compact on $\cB_{\gc}$ with 
$$\rc(P)=\dc := \phi(\gc)<1$$
where $\phi(\cdot)$ is given by (\ref{non-hom-cont}). 

For any $\lambda\in \C$, we denote by $E_{\lambda}(\cdot)$ the following polynomial of degree $N:=d+g$ 
$$
\forall z\in\C,\quad E_{\lambda}(z) := z^g\big(\phi(z) -\lambda\big) = \sum_{k=-g}^{d} a_{k} z^{g+k} - \lambda\, z^g, 
$$
and by $\cE_\lambda$ the set of complex roots of $E_\lambda(\cdot)$. Since $E_{\lambda}(0)=a_{-g} > 0$, we have for any $\lambda\in\C$: 
\[z \in \cE_\lambda \Longleftrightarrow \E_\lambda(z)=0\, \Longleftrightarrow\, \lambda=\phi(z). \] 
 The next proposition investigates the eigenvalues of $P$ on $\cB_{\gc}$ which belong to the annulus 
$$\Lambda := \{\lambda\in\C: \dc < |\lambda| < 1\}.$$
To that effect, for any $\lambda\in\Lambda$, we introduce the following subset  $\cE_\lambda^{-}$  of $\cE_\lambda$ 
$$\cE_\lambda^{-} := \big\{z\in \C:\ E_{\lambda}(z) = 0,\ |z| < \gc\big\}.$$ 
If $\cE_\lambda^{-}=\emptyset$, we set $N(\lambda):=0$. If $\cE_\lambda^{-}\neq\emptyset$, then $N(\lambda)$ is defined as 
$$N(\lambda) := \sum_{z\in {\cal E}_\lambda^{-}}\,  m_z,$$
where $m_z$ denotes the multiplicity of $z$ as root of $E_{\lambda}(\cdot)$. 
Finally, for any $z\in\C$, we set $z^{(1)}:=\{z^n\}_{n\in\N}$, and for any $k\geq2$, $z^{(k)}\in\C^\N$ is defined by: 
$$\forall n\in\N,\quad z^{(k)}(n) := n(n-1)\cdots (n-k+2)\, z^{n-k+1}.$$
\begin{pro} \label{pro-rate-HRW}
Assume that Assumptions~\emph{(\ref{cond-bound-prob})}-\emph{(\ref{Def_HRW-bis})} and \emph{(\ref{MoySautBorne})} hold true. Then 
$$\exists\,  \eta\geq1,\ \forall \lambda\in \Lambda,\quad N(\lambda) = \eta.$$ 
Moreover the two following assertions are equivalent: 
\begin{enumerate}[(i)]
	\item $\lambda\in \Lambda$ is an eigenvalue of $P$ on $\cB_{\gc}$.
	\item There exists a nonzero $\{\alpha_{\lambda,z,k}\}_{z\in {\cal E}_\lambda^{-},1\leq k \leq m_z}\in\C^{\eta}$ such that 
\begin{equation} \label{form-f}
f := \sum_{z\in {\cal E}_\lambda^{-}}\, \sum_{k=1}^{m_z}\alpha_{\lambda,z,k}\, z^{(k)} \in \C^{\N}
\end{equation}
satisfies the boundary equations: $\forall i=0,\ldots,g-1,\ \lambda f(i) = (Pf)(i)$. 
\end{enumerate}
\end{pro}
The first step of the elimination procedure of Section~\ref{rate-general-procedure} is to plug $f$ of the form (\ref{form-f}) in the boundary equations. This gives a linear system in $\alpha_{\lambda,z,k}$. Since $\Lambda$ is infinite, that $N(\lambda)$ does not depend on $\lambda$ is  
crucial to initialize this procedure. To specify the value of $\eta$, it is sufficient to compute $N(\lambda)$ for some (any)  $\lambda\in\Lambda$. 
\begin{rem} \label{rem-equi-delta}
Under Condition~\emph{(\ref{MoySautBorne})}, $\phi(\cdot)$ is strictly decreasing from $(1,\gc)$ to $(\dc,1)$, so that we have: $\forall \lambda\in(\dc,1),\ \phi^{-1}(\lambda)\in(1,\gc)$. Since $\phi^{-1}(\lambda)\in\cE_\lambda$, we obtain 
\begin{equation} \label{N-geq-1}
\forall \lambda\in(\dc,1),\quad N(\lambda)\geq 1.
\end{equation}
\end{rem}
\begin{rem} \label{rem-mod-z-strict}
Let Condition~\emph{(\ref{MoySautBorne})} be satisfied. Set $\cE_\lambda^{+} := \{z\in \C:\ E_{\lambda}(z) = 0,\ |z| > \gc\}$. Then 
\begin{equation*} \label{mod-z-strict}
\forall \lambda\in\Lambda,\quad \cE_\lambda = \cE_\lambda^{-} \sqcup   \cE_\lambda^{+}.
\end{equation*}
In other words, for any $\lambda\in\Lambda$, $E_\lambda(\cdot)$ has no root of modulus $\gc$. Indeed, consider $\lambda\in \Lambda$, $z\in \cE_\lambda$, and assume that $|z| = \gc$. Since $\lambda=\phi(z)$, we obtain the inequality $|\lambda| \leq \phi(|z|)= \phi(\gc)$ which is impossible since $\phi(\gc)=\dc$ and  $\lambda\in \Lambda$.   
\end{rem}
\begin{rem} \label{rem-dim}
Assertion $(ii)$ of Proposition~\ref{pro-rate-HRW} does not mean that the dimension of the eigenspace $\ker(P-\lambda I)$ associated with $\lambda$ is $\eta$. We shall see in Subsection~\ref{further-ex}  that we can have $\eta=2$ when $g=2$, $d=1$ and $c=2$ in \emph{(\ref{cond-bound-prob})}-\emph{(\ref{Def_HRW-bis})}, while $\dim\ker(P-\lambda I)\leq 1$ since $Pf=\lambda f$ and $f(0)=0$ clearly imply $f=0$ (by induction).  
\end{rem}

The following surprising lemma, based on Remark~\ref{rem-mod-z-strict}, is used to derive Proposition~\ref{pro-rate-HRW}.
\begin{lem} \label{fct-N-cste}
Under Condition~\emph{(\ref{MoySautBorne})}, the function $N(\cdot)$ is constant on $\Lambda$. 
\end{lem}
\begin{proof}{}
Since $\Lambda$ is connected and $N(\cdot)$ is $\N$-valued, it suffices to prove that $N(\cdot)$ is continuous on $\Lambda$. Note that the set $\cup_{\lambda\in\Lambda}\cE_{\lambda}$ is bounded in $\C$ since the coefficients of $E_\lambda(\cdot)$ are obviously uniformly bounded in $\lambda\in\Lambda$.  Now let $\lambda\in \Lambda$ and assume that $N(\cdot)$ is not continuous at $\lambda$. Then there exists a sequence $\{\lambda_n\}_{n\in\N}\in \Lambda^\N$ such that $\lim_n \lambda_n = \lambda$ and  
\begin{enumerate}[(a)]
	\item either: $\ \forall n\geq 0,\ N(\lambda_n) \geq N(\lambda)+1$,
	\item or: $\ \forall n\geq 0,\ N(\lambda_n) \leq N(\lambda)-1$. 
\end{enumerate}
For any $n\geq 0$, let us denote the roots of $E_{\lambda_n}(\cdot)$ by $z_1(\lambda_n),\ldots,z_N(\lambda_n)$, and suppose for convenience that they are listed by increasing modulus, and by increasing argument when they have the same modulus. Applying Remark~\ref{rem-mod-z-strict} to $\lambda_n$, we obtain:  
\begin{equation*} \label{N-lambda-bis}
\forall i\in \{1,\ldots,N(\lambda_n)\},\ |z_i(\lambda_n)| < \gc \quad \text{and} \quad \forall i\in \{N(\lambda_n)+1,\ldots, N\},\ |z_i(\lambda_n)| > \gc. 
\end{equation*}
Up to consider a subsequence, we may suppose that, for every $1\leq i \leq N$, the sequence $\{z_i(\lambda_n)\}_{n\in\N}$ converges to some $z_i\in\C$. Note that 
$$\cE_\lambda = \{z_1,z_2,\ldots,z_N\}$$
where $z_i$ is repeated in this list with respect to its multiplicity $m_{z_i}$, since   
$$\forall z\in\C,\quad E_{\lambda}(z) = \lim_n E_{\lambda_n}(z) = \lim_n  a_d\prod_{i=1}^N(z-z_i(\lambda_n)) = a_d\prod_{i=1}^N(z-z_i).$$
In case $(a)$, we have  
$$\forall n\geq0,\qquad |z_1(\lambda_n)| < \gc,\ \ldots,\ |z_{N(\lambda)+1}(\lambda_n)| < \gc.$$ 
When $n\r+\infty$, this gives using Remark~\ref{rem-mod-z-strict}: 
$$|z_1| < \gc\ ,\ldots,|z_{N(\lambda)+1}| < \gc.$$
Thus at least $N(\lambda)+1$ roots of $E_\lambda(\cdot)$ (counted with their multiplicity) are of modulus strictly less than $\gc$: this contradicts the definition of  $N(\lambda)$. \newline 
In case $(b)$, we have
$$\forall n\geq0, \qquad |z_{N(\lambda)}(\lambda_n)| > \gc,\  |z_{N(\lambda)+1}(\lambda_n)| > \gc,\ \ldots,\ |z_{N}(\lambda_n)| > \gc,$$
and this gives similarly when $n\r+\infty$ 
$$|z_{N(\lambda)}| > \gc,\  |z_{N(\lambda)+1}| > \gc,\ldots,|z_{N}| > \gc.$$
Thus at least $N-N(\lambda)+1$ roots of $E_\lambda(\cdot)$ (counted with their multiplicity) are of modulus strictly larger than $\gc$. This contradicts the definition of  $N(\lambda)$. 
\end{proof}
\begin{proof}{ of Proposition~\ref{pro-rate-HRW}}
From Lemma~\ref{fct-N-cste} and (\ref{N-geq-1}), we obtain: $\forall \lambda\in \Lambda,\ N(\lambda) = \eta $ for some $\eta\geq 1$. Now we prove the implication $(i)\Rightarrow (ii)$. Let $\lambda\in \Lambda$ be any eigenvalue of $P$ on $\cB_{\gc}$ and let $f:=\{f(n)\}_{n\in\N}$ be a nonzero sequence in $\cB_{\gc}$ satisfying $Pf = \lambda f$. In particular $f$ satisfies the following equalities     
\begin{equation} \label{rel-rec}
\forall i\geq g,\quad \lambda\, f(i) = \sum_{j=i-g}^{i+g} a_{j-i}f(j).
\end{equation}
Since the characteristic polynomial associated with these recursive formulas is $E_\lambda(\cdot)$, there exists $\{\alpha_{\lambda,z,k}\}_{z\in {\cal E}_\lambda,1\leq k \leq m_z}\in\C^{\eta}$ such that 
$$f = \sum_{z\in{\cal E}_\lambda}\sum_{k=1}^{m_z} \alpha_{\lambda,z,k}\, z^{(k)} \in \C^{\N}$$ 
where $m_z$ denotes the multiplicity of $z\in{\cal E}_\lambda$. Next, since $|f| \leq C\, V_{\gc}$ for some $C>0$ (i.e.~$f\in\cB_{\gc}$), it can be easily seen that $\alpha_{\lambda,z,k}=0$ for every $z\in{\cal E}_\lambda$ such that $|z| > \gc$ and for every $k=1,\ldots,m_z$: : first delete $\alpha_{\lambda,z,m_z}$ for $z$ of maximum modulus and for $m_z$ maximal if there are several $z$ of maximal modulus (to that effect, divide $f$ by $n(n-1)\cdots(n-m_z+2)\, z^{n-m_z+1}$ and use $|f| \leq C V_{\hat\gamma})$. Therefore $f$ is of the form (\ref{form-f}), and it satisfies the boundary equations in (ii) since $Pf = \lambda f$ by hypothesis. 

To prove the implication $(ii)\Rightarrow (i)$, note that any $f:=\{f(n)\}_{n\in\N}$ of the form (\ref{form-f}) belongs to $\cB_{\gc}$ and satisfies (\ref{rel-rec}) since $\cE_\lambda^{-}\subset\cE_\lambda$. If moreover $f$ is non zero and satisfies the boundary equations, then $Pf = \lambda f$. This gives $(i)$. 
\end{proof}

We conclude this study with an additional refinement of Proposition~\ref{pro-rate-HRW}. For any $\lambda\in\Lambda$, let us define the set $\cE_{\lambda,\tau}^{-}$ as follows: 
$$\cE_{\lambda,\tau}^{-} := \big\{z\in \C:\ E_{\lambda}(z) = 0,\ |z| < \gc^\tau\big\} \quad \text{ with } \quad \tau \equiv \tau(\lambda) := \frac{\ln|\lambda|}{\ln\dc}.$$ 
Moreover define the associated function $N'(\cdot)$ by
$$N'(\lambda) := \sum_{z\in {\cal E}_{\lambda,\tau}^{-}}\,  m_z,$$
where $m_z$ is the multiplicity of $z$ as root of $E_{\lambda}(\cdot)$ (with the convention $N'(\lambda)=0$ if $\cE_{\lambda,\tau}^{-}=\emptyset$). 

\begin{lem} \label{rem-ext-tau}
Assume that $P:=(P(i,j))_{(i,j)\in\N^2}$ satisfies Conditions~\emph{(\ref{cond-bound-prob})-(\ref{Def_HRW-bis})}  and  \emph{(\ref{MoySautBorne})}. Moreover assume that 
\begin{equation}
\forall t\in(1,\gc),\quad \phi(t) < t^{\ln\dc/\ln\gc} \label{psi-nega} 
\end{equation}
Then the function $N'(\cdot)$ is constant on $\Lambda$: $\, \exists\,  \eta'\geq1,\ \forall \lambda\in \Lambda,\ N'(\lambda) = \eta'$. 
\end{lem}
From Lemma~\ref{rem-ext-tau}, all the assertions of Proposition~\ref{pro-rate-HRW} are still valid when $\eta$ and $\cE_\lambda^{-}$ are replaced with $\eta'$ and $\cE_{\lambda,\tau}^{-}$ respectively. That $\cE_\lambda^{-}$ may be replaced with $\cE_{\lambda,\tau}^{-}$ in (\ref{form-f}) follows from Proposition~\ref{pro-tail-fct-propre}. Consequently, under the additional condition $\eta' \leq g$, the elimination procedure of Section~\ref{rate-general-procedure} may be adapted by using Lemma~\ref{rem-ext-tau}. Since $\eta'\leq\eta$, the resulting  procedure is computationally interesting when $g$ or $d$ are large.   
\begin{rem} \label{rem-equi-psi-nega-bis}
Condition~\emph{(\ref{psi-nega})} is the additional assumption in Lemma~\ref{rem-ext-tau} with respect to Lemma~\ref{fct-N-cste}. Since $\phi$ is strictly decreasing on $(1,\gc)$ under Condition~\emph{(\ref{MoySautBorne})}, Condition~\emph{(\ref{psi-nega})} is equivalent to the following one 
\begin{equation} \label{equi-psi-nega}
\forall z\in(1,\gc),\quad z < \gc^{\ln\phi(z)/\ln\dc}.
\end{equation}
Indeed, for every $t\in(1,\gc)$, we have $u:=t^{\ln\dc/\ln\gc}\in(\dc,1)$ and $z:=\phi^{-1}(u)\in(1,\gc)$. Hence 
\begin{equation} \label{equi-psi-nega-bis}
(\ref{psi-nega})\ \Longleftrightarrow\ \forall u\in(\dc,1),\ \phi\big(\gc^{\ln u/\ln\dc}\big) < u\ \Longleftrightarrow\ (\ref{equi-psi-nega}). 
\end{equation}
Therefore, under Condition~\emph{(\ref{psi-nega})}, for any $\lambda\in(\dc,1)$ we have $\cE_{\lambda,\tau}^{-}\neq\emptyset$ since $z=\phi^{-1}(\lambda)$ satisfies $z < \gc^{\, \tau(\lambda)}$ from \emph{(\ref{equi-psi-nega})}. 
\end{rem}
\begin{proof}{ of Lemma~\ref{rem-ext-tau}}
The proof is similar to that of Lemma~\ref{fct-N-cste}. Under Condition~(\ref{psi-nega}), Remark~\ref{rem-mod-z-strict} extends as follows: 
\begin{equation} \label{mod-z-strict-tau}
\cE_\lambda = \cE_{\lambda,\tau}^{-} \sqcup   \big(\cE_\lambda\cap\big\{z\in \C:\, |z| > \gc^\tau\big\}\big).
\end{equation}
Indeed, consider $\lambda\in \Lambda$ and $z\in \cE_\lambda$ such that $|z| = \gc^{\tau}$. Since $\lambda=\phi(z)$, we have $|\lambda| \leq \phi(|z|)$, thus $|\lambda| \leq \phi(\gc^{\tau})$. 
This inequality contradicts Condition~(\ref{psi-nega}) (use the definition of $\tau$ and the second equivalence in (\ref{equi-psi-nega-bis}) with $u:=|\lambda|$). 
Next, using (\ref{mod-z-strict-tau}) and the continuity of $\tau(\cdot)$, Lemma~\ref{fct-N-cste} easily extends to the function $N'(\cdot)$. 
\end{proof}

\section{Convergence rate for RWs with i.d.~bounded increments}\label{rate-general-procedure}

Let us recall that any RW~with i.d.~bounded increments defined by (\ref{cond-bound-prob})-(\ref{Def_HRW-bis}) and satisfying (\ref{MoySautBorne}), has an invariant probability measure $\pi$ on $\N$ such $\pi(V_{\gc})<\infty$ where $V_{\gc}:=\{\gc^n\}_{n\in\N}$ and $\gc$ is defined in Lemma~\ref{NERI}.   Indeed $\dc:=\phi(\gc)<1$ so that Condition~(\ref{cond-D}) holds with $V_{\gc}$ from Proposition~\ref{pro-non-hom}. The expected conclusions on $\pi$ can be deduced from the first statement of \cite[Cor~5]{GuiHerLed11}. Note that, from Lemma~\ref{NERI}, the previous fact is valid for any $\gamma\in(1,\gamma_0)$ in place of $\gc$.

The $V_{\gc}$-geometric ergodicity of the RW may be studied using Proposition~\ref{CNS-qc-Vgeo}. Next we can derive from Proposition~\ref{pro-rate-HRW} an effective procedure to compute the rate of convergence with respect to $\cB_{\gc}$ (see (\ref{Def_rhoV})), that is denoted by $\rhoc(P)$. The most favorable case for initializing the procedure (see (\ref{eli-linear-prim}) and  (\ref{eli-linear-sec})) is to assume that  for some (any) $\lambda\in\Lambda$
\begin{equation} 
\eta := N(\lambda) \leq g. \label{cond-card-d}
\end{equation}
\leftmargini 1.5em
\begin{itemize}
	\item {\it First step: checking Condition~\emph{(\ref{cond-card-d})}.}
 From Lemma~\ref{fct-N-cste}, computing $\eta$ and testing $\eta \leq g$ of Assumption~(\ref{cond-card-d}) can be done by analyzing the roots of $E_{\lambda}(\cdot)$ for some (any) $\lambda\in\Lambda$.
	\item {\it Second step: linear and polynomial eliminations.} This second step consists in applying some linear and (successive) polynomial eliminations in order to find a finite set ${\cal Z}\subset\Lambda$ containing all the eigenvalues of $P$ on $\cB_{\gc}$ in $\Lambda$. Conversely, the elements of ${\cal Z}$ providing eigenvalues of $P$ on $\cB_{\gc}$ can be identified using Condition~$(ii)$ of Proposition~\ref{pro-rate-HRW}. Note that the explicit computation of the roots of $E_\lambda(\cdot)$ is only required for the elements $\lambda$ of the finite set ${\cal Z}$. This is detailed in Corollary~\ref{cor-step-3}.  
\end{itemize}

Under the assumptions of Proposition~\ref{pro-rate-HRW}, we define the set 
$$\cM := \big\{ (m_1,\ldots,m_s)\in\{1,\ldots,s\}^s : s\in\{1,\ldots,\eta\}, m_1\leq\ldots\leq m_s \text{ and }\sum_{i=1}^sm_i=\eta\big\}.$$
 Note that  $\cM$ is a finite set and that, for every $\lambda\in \Lambda$, there exists a unique $\mu\in\cM$ such that the set ${\cal E}_\lambda^{-}$ is composed of $s$ distinct roots of $E_\lambda(\cdot)$ with multiplicity $m_1,\ldots,m_s$ respectively. 
\begin{cor} \label{cor-step-3}
Assume that Assumptions~\emph{(\ref{cond-bound-prob})}-\emph{(\ref{Def_HRW-bis})} and \emph{(\ref{MoySautBorne})} hold true. Set  $\ell:= {g\choose\eta}$. Then there exist a family of polynomials functions $\{\cR_{\mu,k},\mu\in\cM,\, 1\leq k\leq\ell\}$, with coefficients only depending on $\mu$ and on the transition probabilities $P(i,j)$,  such that the following assertions hold true for any $\mu\in\cM$. 
\begin{enumerate}[(i)]
	\item Let $\lambda\in \Lambda$ be an eigenvalue of $P$ on $\cB_{\gc}$ such that, for some $s\in\{1,\ldots,\eta\}$, the set ${\cal E}_\lambda^{-}$ is composed of $s$ roots of $E_\lambda(\cdot)$ with multiplicity $m_1,\ldots,m_s$ respectively. Then   
\begin{equation} \label{cond-eli}
\cR_{\mu,1}(\lambda)=0,\ldots, \cR_{\mu,\ell}(\lambda)=0.
\end{equation}
	\item Conversely, let $\lambda\in\Lambda$ satisfying (\ref{cond-eli}) such that, for some $s\in\{1,\ldots,\eta\}$, the set ${\cal E}_\lambda^{-}$ is composed of $s$ roots of $E_\lambda(\cdot)$ with multiplicity $m_1,\ldots,m_s$ respectively. Then a necessary and sufficient condition for  $\lambda$ to be an eigenvalue of $P$ on $\cB_{\gc}$ is that $\lambda$ satisfies Condition~$(ii)$ of Proposition~\ref{pro-rate-HRW}. 
\end{enumerate}
\end{cor}
\begin{proof}{}
Assertion \emph{(ii)} follows from Proposition~\ref{pro-rate-HRW}. To prove \emph{(i)}, first assume for convenience that $\eta=g$ and that $\lambda\in \Lambda$ is an eigenvalue of $P$ on $\cB_{\gc}$ such that the associated set ${\cal E}_\lambda^{-}$ contains $\eta$ distinct roots $z_1,\ldots,z_{\eta}$ of $E_\lambda(\cdot)$ with multiplicity one. We know from Proposition~\ref{pro-rate-HRW} that there exists $f:=\{f(n)\}_{n\in\N}\neq0$ of the form 
\begin{equation*} \label{form-f-simple}
f = \sum_{i=1}^\eta \alpha_{i}\, z_i^{(1)}
\end{equation*}
which satisfies the $g=\eta$ boundary equations: $\forall i=0,\ldots,\eta-1,\ \lambda f(i) = (Pf)(i)$. In other words the linear system provided by these $\eta$ equations has a nonzero solution $(\alpha_{i})_{1\leq i \leq \eta}\in\C^{\eta}$. Therefore the associated determinant is zero: this leads to a polynomial equation of the form 
\begin{equation} \label{eli-linear-prim}
P_{0,1}(\lambda,z_1,\ldots,z_{\eta}) = 0. 
\end{equation}
Since this polynomial is divisible by $\prod_{i\neq j}(z_i-z_j)$, Equation~(\ref{eli-linear-prim}) is equivalent to 
\begin{equation} \label{eli-linear}
P_0(\lambda,z_1,\ldots,z_{\eta})=0\qquad \text{with}\ 
 P_0(\lambda,z_1,\ldots,z_{\eta}) = \frac{P_{0,1}(\lambda,z_1,\ldots,z_{\eta})}{\prod_{i\neq j}(z_i-z_j)}. 
\end{equation}
Note that the coefficients of $P_{0}$ only depend on the $P(i,j)$'s. 

Next, $z_{\eta}$ is a common root of the polynomials $P_{0}(\lambda,z_1,\ldots,z_{{\eta}-1},z)$ and $E_{\lambda}(z)$ with respect to the variable $z$ : this leads to the following necessary condition 
$$P_{1}(\lambda,z_1,\ldots,z_{{\eta}-1}) := \mathrm{Res}_{z_{\eta}}(P_{0},E_{\lambda}) = 0$$
where $\mathrm{Res}_{z_{\eta}}(P_{0},E_{\lambda})$ denotes the resultant of the two polynomials $P_0$ and $E_{\lambda}$ corresponding to the elimination of the variable $z_{\eta}$. 
Again the coefficients of $P_{1}$ only depend on the $P(i,j)$'s. Next, considering the common root $z_{{\eta}-1}$ of the polynomials $P_{1}(\lambda,z_1,\ldots,z_{{\eta}-2},z)$ and $ E_{\lambda}(z)$ leads to the elimination of the variable $z_{{\eta}-1}$
$$P_{2}(\lambda,z_1,\ldots,z_{{\eta}-2}) := \mathrm{Res}_{z_{{\eta}-1}}(P_{1},E_{\lambda}) = 0.$$
Repeating this method, we obtain that a necessary condition for $\lambda$ to be an eigenvalue of $P$ is $\cR(\lambda)=0$ where $\cR$ is some polynomial with coefficients only depending on the $P(i,j)$'s. 

Now let us consider the case when $\eta < g$, $s\in\{1,\ldots,\eta\}$, and $\lambda\in \Lambda$ is assumed to be an eigenvalue of $P$ on $\cB_{\gc}$ such that the associated set ${\cal E}_\lambda^{-}$ contains $s$ distinct roots of $E_\lambda(\cdot)$ with respective multiplicity $m_1,\ldots,m_s$ satisfying $\sum_{i=1}^sm_i=\eta$. Then the elimination (by using determinants) of $(\alpha_{\lambda,z,\ell})\in\C^{\eta}$ provided by the linear system of Proposition~\ref{pro-rate-HRW}, leads to $\ell:= {g\choose\eta}$ polynomial equations 
\begin{equation} \label{eli-linear-sec}
P_{0,\mu,1}(\lambda,z_1,\ldots,z_\eta)=0,\ \ldots,\  P_{0,\mu,\ell}(\lambda,z_1,\ldots,z_\eta)=0.
\end{equation}
As in the case $\eta=g$, these polynomials are replaced in the sequel by the  polynomials obtained by division of the $P_{0,\mu,k}$'s by $\prod_{i\neq j} (z_i-z_j)^{n_{i,j}}$ where $n_{i,j}:=\min(m_i,m_j)$.

The successive polynomial eliminations of $z_\eta,\ldots,z_1$ can be derived as above from each polynomial equation $P_{0,\mu,k}(\lambda,z_1,\ldots,z_\eta)=0$. This gives $\ell$ polynomial equations 
$$\cR_{\mu,1}(\lambda)=0\ ,\ldots,\ \cR_{\mu,\ell}(\lambda)=0.$$
Satisfying this set of polynomial equations is a necessary condition for $\lambda$ to be an eigenvalue of $P$ on $\cB_{\gc}$. Finally the polynomial functions $\cR_{\mu,1},\ldots\cR_{\mu,\ell}$ depend on the $P(i,j)$'s and also on $(m_1,\ldots,m_s)$, since the linear system used to eliminate $(\alpha_{\lambda,k,\ell})\in\C^{\eta}$ involves coefficients $i(i-1)\cdots (i-k+1)$ for some finitely many integers $i$ and for $k=1,\ldots,m_i$ ($i=1,\ldots,s$). 
\end{proof}

To compute $\rhoc(P)$, we define the following (finite and possibly empty) sets:
$$\forall \mu\in\cM,\quad \Lambda_\mu:=\big\{\lambda\in\Lambda :\ \cR_{\mu,1}(\lambda)=0\ ,\ldots,\ \cR_{\mu,\ell}(\lambda)=0\big\}.$$
Let us denote by ${\cal Z}$ the (finite and possibly empty) set composed of all the complex numbers $\lambda\in\cup_{\mu\in{\cal M}} \Lambda_\mu$ such that Condition~$(ii)$ of Proposition~\ref{pro-rate-HRW} holds true.  
\begin{cor} \label{cor-rate}
Assume that Assumptions~\emph{(\ref{cond-bound-prob})}-\emph{(\ref{Def_HRW-bis})} and \emph{(\ref{MoySautBorne})} hold true and that $P$ is irreducible and aperiodic.   Then   
$$\rhoc(P) = \max\big(\dc,\max\{|\lambda|,\ \lambda\in {\cal Z}\}\big) \quad \text{where $\dc:=\phi(\gc)$}.$$
\end{cor}
\begin{proof}{}
Under the assumptions on $P$, we know from Proposition~\ref{CNS-qc-Vgeo} that the RW is $V_{\gc}$-geometrically ergodic.  
Since  $\widehat{r}_{ess}(P)=\dc$ from Proposition~\ref{pro-non-hom}, the corollary follows from Corollary~\ref{cor-step-3} and from Proposition~\ref{CNS-qc-Vgeo} applied either with any $r_0$ such that $\dc<r_0<\min\{|\lambda|,\ \lambda\in {\cal Z}\}$ if ${\cal Z}\neq\emptyset$, or with any $r_0$ such that $\dc<r_0<1$ if ${\cal Z}=\emptyset$. 
\end{proof}

\begin{rem} \label{rem-proc} 
When $\eta\geq2$ and $\mu:=(m_1,\ldots,m_s)$ with $s<\eta$, the set $\Lambda_\mu$ used in Corollary~\ref{cor-rate} may be reduced.  For the sake of simplicity, this fact has been omitted in Corollary~\ref{cor-rate}, but it is relevant in practice. Actually, when $s<\eta$, the  part $(ii)$ of Corollary~\ref{cor-step-3} can be specified since it requires that $E_\lambda(\cdot)$ admits roots of multiplicity $\geq2$. This involves some additional necessary conditions on $\lambda$ derived from some polynomial eliminations with respect to the derivatives of $E_\lambda(\cdot)$. 

For instance, in case $g=2$, $\eta=2$, $s=1$ (thus $\mu:=(2)$), a necessary condition on $\lambda$ for $E_\lambda(\cdot)$ to have a double root is that $E_\lambda(\cdot)$ and $E_\lambda'(\cdot)$ admits a common root. This leads to 
$$Q(\lambda) := \mathrm{Res}_z\big(E_\lambda,E_\lambda'\big)=0.$$
Consequently, if $g=2$ and $\eta=2$ (thus $\ell:=1$), then  Condition~$(ii)$ of Proposition~\ref{pro-rate-HRW} can be tested in case $s=1$ by using the following finite set 
$$\Lambda_\mu':= \Lambda_\mu \cap \{\lambda\in\Lambda :\ Q(\lambda)=0\}.$$
In general $\Lambda_\mu'$ is strictly contained in $\Lambda_\mu$. Even $\Lambda_\mu'$ may be empty while $\Lambda_\mu$ is not (see Subsection~\ref{further-ex}). 
\end{rem}

Proposition~\ref{pro-rate-HRW} and the above elimination procedure obviously extend to any $\gamma\in(1,\gamma_0)$ in place of $\gc$, where $\gamma_0$ is given in Lemma~\ref{NERI}. Of course $\dc = \phi(\gc)$ is then replaced by $\delta = \phi(\gamma)$.

\subsection{RWs with $g=d:=1$ : birth-and-death Markov chains} \label{sub-basic-ex-revis} 

%
Let $p,q,r\in[0,1]$ be such that $p+r+q=1$, and let $P$ be defined  by
\begin{equation} \label{hyp-ref-rand}
\begin{array}{c}
 P(0,0)\in(0,1), P(0,1)=1-P(0,0)\\[1mm]
  \forall n\ge 1,\ P(n,n-1) :=p,\quad P(n,n) := r,\quad P(n,n+1) := q \quad \text{with } 0<q<p. 
 \end{array}
\end{equation} 
 Note that $a_{-1}:=p,a_1:=q>0$ and (\ref{MoySautBorne}) holds true. We have $\gamma_0=p/q\in(1,+\infty)$ and $\gc:= \sqrt{p/q}\in(1,+\infty)$ is such that $\dc:=\min_{\gamma>1}\phi(\gamma)=\phi(\gc)<1$ (see Lemma~\ref{NERI}). Let $V_{\gc}:=\{\gc^n\}_{n\in\N}$ and its associated weighted-supremum space $\cB_{\gc}$. Here we have 
$$\widehat{r}_{ess}(P) = \dc = r+2\sqrt{pq}.$$
\begin{pro} \label{BD} 
Let $P$ be defined by Conditions~\emph{(\ref{hyp-ref-rand})}. The boundary transition probabilities are denoted by $P(0,0):=a, P(0,1):=1-a$ for some $a\in(0,1)$. 
Then $P$ is $V_{\gc}$-geometrically ergodic. Furthermore, defining  $a_0 := 1-q-\sqrt{pq}$, the convergence rate $\rhoc(P)$ of $P$ with respect to $\cB_{\gc}$ is given by: 
\leftmargini 1.5em   
\begin{itemize}
	\item when $a\in(a_0,1)$: 
  \begin{equation} \label{a0-2}
  	\rhoc(P) = r+2\sqrt{pq}\, ; 
  \end{equation} 
	\item when $a\in(0,a_0]$: 
\begin{enumerate}[(a)]
	\item in case $\, 2p \leq \big(1-q+\sqrt{pq}\big)^2$:
\begin{equation}\label{a0-0}
\rhoc(P) = r+2\sqrt{pq}\, ; 
\end{equation} 
  \item in case $\, 2p > \big(1-q+\sqrt{pq}\big)^2$, set $a_1 := p - \sqrt{pq}  - \sqrt{r\big(r+2\sqrt{pq}\big)}$:
\begin{subequations}
\begin{eqnarray}
& &  \rhoc(P) = \left|a + \frac{p(1-a)}{a-1+q}\right| \ \ \text{ when } a\in(0,a_1] \label{a0-1} \\
& & \rhoc(P) =  r+2\sqrt{pq}  \ \, \quad\qquad\text{ when } a\in[a_1,a_0). \label{a0-3}
\end{eqnarray}
\end{subequations}
\end{enumerate}
\end{itemize}
\end{pro}
When $r:=0$, such results have been obtained in \cite{RobTwe99,Bax05,LunTwe96} by using various methods involving conditions on $a$ (see the end of Introduction). Let us specify the above formulas in case $r:=0$. We have $a_0 = a_1 = p-\sqrt{pq} = (p-q)/(1+\sqrt{q/p})$, and it can be easily checked that $2p > (1-q+\sqrt{pq})^2$. Then the properties (\ref{a0-2}), (\ref{a0-1}), (\ref{a0-3}) then rewrite as: 
$\rhoc(P) = (pq+(a-p)^2)/|a-p|$ when $a\in(0,a_0]$, and $\rhoc(P) = 2\sqrt{pq}$ when $a\in(a_0,1)$. 

\begin{proof}{}
We apply the elimination procedure of Section~\ref{rate-general-procedure}. 
Then $\Lambda:=\{\lambda\in\C:\, \dc < |\lambda| < 1\}$ with $\dc:=r+2\sqrt{pq}$. The characteristic polynomial $E_\lambda(\cdot)$ is 
$$E_\lambda(z) := qz^2 + (r - \lambda) z + p.$$
A simple study of the graph of $\phi(t) := p/t+r+qt$ on $\R\setminus\{0\}$ shows that, for any $\lambda\in (\dc,1)$, the equation $\phi(z)=\lambda$ (ie.~$E_\lambda(z)=0$) admits a solution in $(1,\gc)$ and another one in $(\gc,+\infty)$, so that $N(\lambda)=1$. It follows from Proposition~\ref{pro-rate-HRW} that $\eta=1$.  
Thus the linear elimination used in Corollary~\ref{cor-step-3} is here trivial. Indeed, a necessary condition for $f:=\{z^n\}_{n\in\N}$ to satisfy $Pf=\lambda f$ is obtained by eliminating the variable $z$ with respect to the boundary equation $(Pf)(0) = \lambda f(0)$, namely $P_0(\lambda,z):=a+(1-a)z = \lambda$, and Equation $E_\lambda(z) = 0$. This leads to 
\begin{eqnarray} 
& & P_1(\lambda,z):=Res_{z}(P_0,E_{\lambda})=(1-\lambda)\big[(\lambda-a)(1-a-q) +p(1-a) \big] \label{eq-lambda-glob}.
\end{eqnarray}
In the special case $a = 1-q$,  the only solution of (\ref{eq-lambda-glob}) is $\lambda=1$. Corollary~\ref{cor-rate} then gives $\rhoc(P) = r+2\sqrt{pq}$. 

Now assume that $a \neq 1-q$. Then $\lambda=1$ is a solution of (\ref{eq-lambda-glob}) and the other solution of (\ref{eq-lambda-glob}), say $\lambda(a)$, and the associated complex number, say $z(a)$, are given by the following formulas (use $a+(1-a)z = \lambda$ to obtain $z(a)$): 
\begin{equation*} \label{lambda=}
\lambda(a) :=  a + \frac{p(1-a)}{a-1+q}\in\R \quad \text{ and } \quad  z(a) := \frac{p}{a+q-1}\in\R.
\end{equation*}
To apply Corollary~\ref{cor-rate} we must  find the values $a\in(0,1)$ for which both conditions $\dc < |\lambda(a)| < 1$ and $|z(a)|\leq \gc$ hold. Observe that
\begin{equation*} \label{machin}
|z(a)|\leq \gc  \ \Leftrightarrow\ |a-1+q|\geq \sqrt{pq}. 
\end{equation*}
Hence, if $a\in(a_0,1)$ (recall that $a_0 := 1-q-\sqrt{pq}$), then $|z(a)| >\gc$. This gives (\ref{a0-2}). 

Now let $a\in(0,a_0]$. Then $|z(a)| \leq \gc$. Let us study $\lambda(a)$. 
We have $\lambda'(a) = 1-pq/(a-1+q)^2$, so that 
$a\mapsto \lambda(a)$ is increasing on $(-\infty,a_0]$ from $-\infty$ to $\lambda(a_0)=r-2\sqrt{pq}$.  
Thus
$$\forall a\in(0,a_0],\quad \lambda(a)\leq r-2\sqrt{pq} < r+2\sqrt{pq}.$$ 
and the equation $\lambda(a) = -(r+2\sqrt{pq})$ has a unique solution $a_1\in(-\infty,a_0)$. 
Note that $a_1<a_0$ and $\lambda(a_1) = -(r+2\sqrt{pq})$, that $\lambda(0) = p/(q-1)\in[-1,0)$ and finally that 
$$\lambda(0) - \lambda(a_1) = \frac{p}{q-1} +r+2\sqrt{pq} = \frac{(q-\sqrt{pq}-1)^2 - 2p}{1-q}.$$
When $2p \leq (1-q+\sqrt{pq})^2$, we obtain (\ref{a0-0}). Indeed $|\lambda(a)| < r+2\sqrt{pq}$ since 
$$\forall a\in(0,a_0],\quad -(r+2\sqrt{pq}) = \lambda(a_1) \leq \lambda(0) < \lambda(a) < r+2\sqrt{pq}.$$
When $2p > (1-q+\sqrt{pq})^2$, we have $a_1\in(0,a_0]$ and:  
\leftmargini 1.5em
\begin{itemize}
	\item if $a\in(0,a_1)$, then (\ref{a0-1}) holds. Indeed $r+2\sqrt{pq} < |\lambda(a)| < 1$ since 
	$$\forall a\in(0,a_1],\quad -1 \leq \lambda(0) < \lambda(a) < \lambda(a_1) = -(r+2\sqrt{pq})\, ;$$ 
	\item if $a\in[a_1,a_0]$, then (\ref{a0-3}) holds. Indeed  $|\lambda(a)| < r+2\sqrt{pq}$ since 
	$$-(r+2\sqrt{pq}) =\lambda(a_1) \leq \lambda(a) < r+2\sqrt{pq}.$$ 
\end{itemize} 
\end{proof}
\begin{rem}[Discussion on the $\ell^2(\pi)$-spectral gap and the decay parameter] \label{l2-spectral-gap} ~\newline
As men\-tio\-ned in the introduction, we are not concerned with the usual $\ell^2(\pi)$ spectral gap $\rho_2(P)$ for Birth-and-Death Markov Chains (BDMC). In particular, we can not compare our results  with that of \cite{DooSch95}. To give a  comprehensive discussion on \cite{DooSch95}, let $P$ be a kernel of an BDMC defined by \emph{(\ref{hyp-ref-rand})} with invariant probability measure $\pi$. $P$ is reversible with respect to $\pi$. It can be proved that the decay parameter of $P$, denoted by $\gamma$ in \cite{DooSch95} but by $\gamma_{DS}$ here to avoid confusion with our parameter $\gamma$, is also the rate of convergence $\rho_2(P)$: 
\begin{equation*} \label{gam-rho}
\gamma_{DS} =\rho_2(P) := \lim_n{\|P^n-\Pi\|_2}^{\frac{1}{n}},
\end{equation*}
where $\Pi f := \pi(f){\bf 1}$ and $\|\cdot\|_2$ denotes the operator norm on $\ell^2(\pi)$. When $P$ is assumed to be $V_{\gc}$-geometrically ergodic with $V:=\{\gc^n\}_{n\in\N}$, it follows from \cite[Th.~6.1]{Bax05},  that 
\begin{equation*}\label{2}
	\gamma_{SD} \leq \rhoc(P).
\end{equation*} 
Consequently the bounds of the decay parameter $\gamma_{DS}$ given in \cite{DooSch95} cannot provide bounds for $\rhoc(P)$ since the converse inequality $\rhoc(P) \leq \gamma_{DS}$ is not known to the best of our knowledge. Moreover, even if the equality $\gamma_{DS} = \rhoc(P)$ was true, the bounds obtained in our Proposition~\ref{BD} could be derived from \cite{DooSch95} only for some specific values of  $P(0,0)$. Indeed the difficulty in \cite[p.~139-140]{DooSch95} to cover all the values $P(0,0)\in(0,1)$ is that the spectral measure associated with Karlin and McGregor polynomials cannot be easily computed, except for some specific values of $P(0,0)$ (see \cite{Kov09} for a recent contribution).  
\end{rem}

\subsection{A non-reversible case : RWs with $g=2$ and $d=1$ } \label{further-ex} 

Let $P:=(P(i,j))_{(i,j)\in\N^2}$ be defined by 
\begin{gather}
P(0,0) = a \in(0,1),\quad P(0,1) = 1-a,\quad P(1,0) = b\in(0,1), \quad P(1,2) = 1-b \label{bound-deux-vois} \\[2mm]
\forall n\geq 2,\ P(n,n-2) = a_{-2}>0,\ P(n,n-1) = a_{-1},\ P(n,n) = a_{0}, \ P(n,n+1) = a_1 >0. \nonumber 
\end{gather}
The form of boundary probabilities in (\ref{bound-deux-vois}) is chosen for convenience. Other (finitely many) boundary probabilities could be considered provided that $P$ is irreducible and aperiodic. To illustrate the procedure proposed in Section~\ref{rate-general-procedure} for this class of RWs, we also specify the numerical values 
$$a_{-2} := 1/2,\ a_{-1} := 1/3,\ a_0 = 0,\ a_1 := 1/6.$$ 
The procedure could be developed in the same way for any other values of $(a_{-2},a_{-1},a_0,a_1)$ satisfying $a_{-2}, a_1 >0$ and Condition (\ref{MoySautBorne}) i.e. $a_1 < 2a_{-2} + a_{-1}$. Here we have 
$$\phi(t) := \frac{1}{2 t^2} + \frac{1}{3t} + \frac{t}{6} = 1+ \frac{1}{6t^2}(t-1)(t^2-5t-3). $$ 
Function $\phi(\cdot)$ has a minimum over $(1,+\infty)$ at  $\gc\approx 2.18$, with $\dc := \phi(\gc) \approx 0.621$. Let $V_{\gc} :=\{\gc^n\}_{n\in\N}$ and let $\cB_{\gc}$ be the associated weighted space. We know from Proposition~\ref{pro-non-hom} and from irreducibility and aperiodicity properties  that $\widehat{r}_{ess}(P) = \dc$ and $P$ is $V_{\gc}\, -$geometrically ergodic (see Proposition~\ref{CNS-qc-Vgeo}). The  polynomial $E_\lambda(\cdot)$ is 
\begin{equation*} \label{def-poly-charac-2vois}
\forall z\in\C,\quad E_{\lambda}(z) := \frac{z^3}{6} - \lambda z^2 + \frac{z}{3} + \frac{1}{2}. 
\end{equation*}
A simple examination of the graph of $\phi(\cdot)$ shows that $\eta=2$.
Thus the set $\cM$ of Corollary~\ref{cor-rate} is $\cM:=\{(1,1),(2)\}$. Next, the constructive proof of Corollary~\ref{cor-step-3} provides the following procedure to compute  $\rhoc(P)$ (see also Remark~\ref{rem-proc} in the second case). Recall that $\Lambda := \{\lambda\in\C: \dc < |\lambda| < 1\}$.

\paragraph{First case: $\mu=(1,1)$}
\begin{enumerate}[(a)]
	\item  When $\lambda\in \Lambda$ is such that ${\cal E}_\lambda^{-}$ is composed of $2$ simple roots of $E_\lambda(\cdot)$, a necessary condition for $\lambda$ to be an  eigenvalue of $P$ on $\cB_{\gc}$ is that 
	$$R_1(\lambda) := \mathrm{Res}_{z_{1}}\big(P_1,E_\lambda\big) = 0,$$
where 
\begin{eqnarray*} 
P_1(\lambda,z_1) : = \mathrm{Res}_{z_{2}}\big(P_{0},E_\lambda\big)
&=& \left|\begin{array}{ccccc}
1/6 &0 & A(\lambda,z_1) & 0 & 0 \\
-\lambda & 1/6 & B(\lambda,z_1) & A(\lambda,z_1) & 0 \\
1/3 & -\lambda & C(\lambda,z_1) & B(\lambda,z_1) & A(\lambda,z_1) \\
1/2 & 1/3 & 0 & C(\lambda,z_1) & B(\lambda,z_1) \\
0 & 1/2 & 0 & 0 & C(\lambda,z_1) 
\end{array}
\right|.  
\end{eqnarray*}
and $P_{0}(\lambda,z_1,z_2) := A(\lambda,z_1)\, {z_2}^2 + B(\lambda,z_1)\, z_2 + C(\lambda,z_1)$  is given by  
\begin{eqnarray}
P_{0}(\lambda,z_1,z_2)  :=\left|\begin{array}{cc}
(1-a) & a+(1-a)z_2-\lambda  \\[0.12cm]
(1-b)(z_1+z_2)-\lambda  & b+(1-b)z_2^2-\lambda z_2   
\end{array}
\right|. \label{P0-mult-1}
\end{eqnarray}
$P_{0}(\lambda,z_1,z_2)$ is derived using (\ref{eli-linear}) from    
\begin{eqnarray*} 
P_{0,1}(\lambda,z_1,z_2) &:=& 
\left|\begin{array}{cc}
a+(1-a)z_1-\lambda & a+(1-a)z_2-\lambda  \\[0.12cm]
b+(1-b)z_1^2-\lambda z_1 & b+(1-b)z_2^2-\lambda z_2   
\end{array}
\right| = (z_1-z_2) P_{0}(\lambda,z_1,z_2). 
\end{eqnarray*} 

	\item {\it Sufficient part.} Consider 
$$\Lambda_{(1,1)} = \text{Root}\, (R_1)\cap\Lambda = \text{Root}\, (R_1)\cap \big\{\lambda\in\C : 0.621\approx \dc< |\lambda|<1\big\}.$$
For every $\lambda\in \Lambda_{(1,1)}$: 
\begin{enumerate}[(i)]
	\item Check that $E_\lambda(z)=0$ has two simple roots $z_1$ and $z_2$ such that $|z_i| < \gc\approx 2.18$.  
	\item If (i) is OK, then test if $P_{0}(\lambda,z_1,z_2)=0$ with $P_{0}$ given in (\ref{P0-mult-1}). 
	\item[] \hspace*{-0.8cm} If (i) and (ii) are OK, then $\lambda$ is an eigenvalue of $P$ on $\cB_{\gc}$. 
\end{enumerate}
\end{enumerate}
\paragraph{Second case: $\mu=(2)$.}
\begin{enumerate}[(a)]
\item {\it }  When $\lambda\in \Lambda$ is such that ${\cal E}_\lambda^{-}$ is composed of a double root of $E_\lambda(\cdot)$, a necessary condition for $\lambda$ to be an  eigenvalue of $P$ on $\cB_{\gc}$ is that (see Remark~\ref{rem-proc})
	$$Q(\lambda) = 0 \quad \text{and} \quad R_2(\lambda) := \mathrm{Res}_{z_{1}}\big(P_1,E_\lambda\big) = 0,$$
	where 
\begin{eqnarray*} 
Q(\lambda) : =
\left|\begin{array}{ccccc}
1/6 &0 & 1/2 & 0 & 0  \\
-\lambda & 1/6 & -2\lambda & 1/2 & 0 \\
1/3 & -\lambda & 1/3 & -2\lambda & 1/2 \\
1/2 & 1/3 & 0 & 1/3 & -2\lambda \\
0 & 1/2 & 0 & 0 & 1/3 \\
\end{array}
\right| 
\end{eqnarray*}
and 
\begin{eqnarray*} 
P_1(\lambda) : = \mathrm{Res}_{z_{1}}\big(P_0,E_\lambda\big)
&=& \left|\begin{array}{ccccc}
1/6 &0 & A(\lambda) & 0 & 0 \\
-\lambda & 1/6 & B(\lambda) & A(\lambda) & 0 \\
1/3 & -\lambda & C(\lambda) & B(\lambda) & A(\lambda) \\
1/2 & 1/3 & 0 & C(\lambda) & B(\lambda) \\
0 & 1/2 & 0 & 0 & C(\lambda) 
\end{array}
\right|.  
\end{eqnarray*}
where $P_0(\lambda,z_1) := \ A(\lambda)\, z_1^2 + B(\lambda)\, z_1 + C(\lambda)$ is given by 
\begin{eqnarray} 
P_0(\lambda,z_1) &:=& 
\left|\begin{array}{cc}
a+(1-a)z_1-\lambda & 1-a  \\[0.12cm]
b+(1-b)z_1^2-\lambda z_1 & 2(1-b)z_1-\lambda 
\end{array}
\right|. \label{P0-mult-2}
\end{eqnarray}

  \item {\it Sufficient part.} Consider 
$$\Lambda_{(2)}' = \text{Root}\, (Q)\cap \Lambda_{(2)} = \text{Root}\, (Q)\cap \text{Root}\, (R_2) \cap \big\{\lambda\in\C : 0.621\approx \dc< |\lambda|<1\big\} .$$
For every $\lambda\in \Lambda_{(2)}'$: 
\begin{enumerate}[(i)]
	\item Check that Equation $E_\lambda(z)=0$ has a double root $z_1$ such that $|z_1| < \gc\approx 2.18$. 
	\item If (i) is OK, then test if $P_0(\lambda,z_1)=0$ with $P_0$ given in (\ref{P0-mult-2}).
	\item[] \hspace*{-0.8cm} If (i) and (ii) are OK, then $\lambda$ is an eigenvalue of $P$ on $\cB_{\gc}$. 
\end{enumerate}
\end{enumerate}
\paragraph{Final results}
Define ${\cal Z}_{(1,1)}$ as the set of all the $\lambda\in \Lambda_{(1,1)}$ satisfying (i)-(ii) in the first case, and ${\cal Z}_{(2)}$ as the set of all the $\lambda\in \Lambda_{(2)}'$ satisfying (i)-(ii) in the second one. Finally set ${\cal Z} := {\cal Z}_{(1,1)}\cup{\cal Z}_{(2)}$. Then  
$$\rhoc(P) = \max\big(\dc,\max\{|\lambda|,\ \lambda\in {\cal Z}\}\big).$$
The results (using Maple computation engine) for different instances of the values of boundary transition probabilities are reported in Table~\ref{Table}. In these specific examples, note that $\Lambda_{(2)}'$ is always the empty set. As expected, we obtain that $\rho_{\gc}(P)\nearrow 1$ when $(a,b)\r (0,0)$. 
\begin{table}[h]
\centering
\begin{tabular}{c||c|c||c|c||c|c} \hline
$(a,b)$ & $\Lambda_{(1,1)}$ & ${\cal Z}_{(1,1)}$ & $\Lambda_{(2)}'$ & ${\cal Z}_{(2)}$ & $\dc$ & $\rhoc(P)$    \\\hline
$(1/2,1/2)$ & $\begin{array}{c}
-0.625\pm 0.466 i, \\
-0.798,0.804 
\end{array}$ & $\emptyset$ & $\emptyset$ & $\emptyset$ & 0.621 & 0.621\\\hline
$(1/10,1/10)$ & $\begin{array}{c}
-0.681\pm 0.610 i \\
-0.466\pm -0.506 i \\
-0.384\pm 0.555 i \\
\end{array}$ & $\{-0.466\pm 0.506 i\}$ & $\emptyset$ & $\emptyset$ & 0.621 & 0.688 \\\hline
$(1/50,1/50)$ & $\begin{array}{c}
-0.598 \pm 0.614 i \\
-0.383\pm 0.542 i  \\
-0.493\pm 0.574 i \\
-0.477\pm 0.584 i \\
0.994
\end{array}$& $\{-0.493\pm 0.574 i\}$ & $\emptyset$ & $\emptyset$ & 0.621 & 0.757\\\hline
\end{tabular}
\caption{Convergence rate with different values of boundary transition probabilities $(a,b)$}
\label{Table}
\end{table}

\section{Convergence rate for RWs with unbounded increments} \label{sect-unbounded} 

In this subsection, we propose two instances of RW on $\X:=\N$ with unbounded increments for which estimate of the convergence rate with respect to some weighted-supremum space $\cB_V$ can be obtained using Proposition~\ref{cor-qc-bis} and Proposition~\ref{CNS-qc-Vgeo}. The first example is from \cite{MalSpi95}. The second one is a reversible transition kernel $P$ inspired from the ``infinite star'' example in \cite{Ros96}. Note that using a result of \cite{Bax05} (see Remark~\ref{l2-spectral-gap}), estimates of $\rho_V(P)$ with respect to $\cB_V$ may be useful to obtain estimates on the usual spectral gap $\rho_2(P)$ with respect to Lebesgue's space $\ell^2(\pi)$. Recall that the converse is not true in general.

\subsection{A non-reversible RW with unbounded increments \cite{MalSpi95}} \label{ex-speksma-1} 
Let $P$ be defined by 
$$\forall n\geq 1,\ P(0,n) := q_n,\quad \forall n\geq 1,\ P(n,0) := p,\ P(n,n+1) := q = 1-p,$$
with $p\in(0,1)$ and $q_n\in[0,1]$ such that $\sum_{n\geq1}q_n=1$. 

\begin{pro} \label{spect-gap-speksma-1}
Assume that $\gamma\in(1,1/q)$ is such that $\sum_{n\geq1}q_n\gamma^n < \infty$. Then $r_{ess}(P) \leq q\gamma$. Moreover $P$ is $V_\gamma$-geometrically ergodic with convergence rate   
$\rho_{V_\gamma}(P) \leq \max(q\gamma,p)$.
\end{pro}
\begin{proof}{}
We have: 
$\forall n\geq1,\ (PV_\gamma)(n) = q\gamma^{n+1} + p$. Thus, if $\gamma\in(1,1/q)$ and $\sum_{n\geq1}q_n\gamma^n < \infty$, then Condition~(\ref{cond-D}) holds with $V_\gamma$, and we have $\delta_{V_\gamma}(P) \leq q\gamma$. Therefore it follows from Proposition~\ref{cor-qc-bis} that $r_{ess}(P) \leq q\gamma$. Now  Proposition~\ref{CNS-qc-Vgeo} is applied with any $r_0>\max(q\gamma,p)$. Let $\lambda\in\C$ be such that $\max(q\gamma,p) < |\lambda| \leq1$, and let $f\in\cB_\gamma$, $f\neq0$, be such that $Pf=\lambda f$. We obtain $f(n) = (\lambda/q)f(n-1) - pf(0)/q$ for any $n\geq 2$, so that 
$$\forall n\geq2,\quad f(n) = \left(\frac{\lambda}{q}\right)^{n-1}\left(f(1)- \frac{pf(0)}{\lambda-q}\right) + \frac{pf(0)}{\lambda-q}.$$
Since $f\in\cB_{V_\gamma}$  and $|\lambda|/q > \gamma$, we obtain $f(1)= pf(0)/(\lambda-q)$, and consequently: $\forall n\geq1,\ f(n)= pf(0)/(\lambda-q)$. Next the equality $\lambda f(0) = (Pf)(0) = \sum_{n\geq1} q_nf(n)$ gives: $\lambda f(0) = pf(0)/(\lambda-q)$ since $\sum_{n\geq1}q_n=1$. We have $f(0)\neq0$ since we look for 
$f\neq0$. Thus $\lambda$ satisfies $\lambda^2 - q\lambda - p=0$, that is: $\lambda = 1$ or $\lambda =-p$. The case $\lambda =-p$ has not to be considered since $|\lambda| > \max(q\gamma,p)$. 
If $\lambda = 1$, then
$f(n)=f(0)$ for any $n\in\N$, so that $\lambda=1$ is a simple eigenvalue of $P$ on $\cB_\gamma$ and is the only eigenvalue 
such that $\max(q\gamma,p) < |\lambda| \leq1$. Then Proposition~\ref{CNS-qc-Vgeo} gives the second conclusion of Proposition~\ref{spect-gap-speksma-1}. 
\end{proof}

Note that $p$ cannot be dropped in the inequality $\rho_{V_\gamma}(P) \leq \max(q\gamma,p)$ since $\lambda=-p$ is an eigenvalue of $P$ on $\cB_\gamma$ with corresponding eigenvector
$f_p:=(1,-p,-p,\dots)$.

\subsection{A reversible RW inspired from \cite{Ros96}} \label{ex-rosen}
Let $\{\pi_n\}_{n\in\N}$ be a probability distribution (with $\pi_n>0$ for every $n\in\N$) and $P$ be  defined by 
$$\forall n\in\N,\ P(0,n)=\pi_n \quad \text{and} \quad  \forall n\ge 1,\  P(n,0)=\pi_0,\ P(n,n)=1-\pi_0.$$
It is easily checked that $P$ is reversible with respect to $\{\pi_n\}_{n\in\N}$, so that $\{\pi_n\}_{n\in\N}$ is an invariant probability distribution of $P$.  
\begin{pro} \label{pro-rhoV-ex}
Assume that there exists $V\in[1,+\infty)^{\N}$ such that $V(0)=1$, $V(n)\r +\infty$ as $n\r+\infty$ and  
 $\pi(V):=\sum_{n\geq0} \pi_n V(n) < \infty$.
Then $P$ is $V$-geometrically ergodic with $\rho_V(P)\leq 1-\pi_0$. 
\end{pro}
It can be checked that $P$ is not stochastically monotone so that the estimate $\rho_V\leq 1-\pi_0$ cannot be directly deduced from \cite{LunTwe96}. 

\noindent\begin{proof}{}
From $(PV)(0) =  \pi(V)$ and $\forall n\geq1,\ (PV)(n) =  \pi_0 V(0) + (1-\pi_0)V(n) $, it follows that  
$$PV \leq (1-\pi_0) V + (\pi(V)+\pi_0)\, 1_{\X}.$$
That is, Condition~(\ref{cond-D}) holds true with $N:=1$, $\delta:=1-\pi_0$ and $d:=\pi(V) + \pi_0$.
The inequality $r_{ess}(P)\leq 1-\pi_0$ is deduced from Proposition~\ref{cor-qc-bis}. 

Let $\lambda\in\C$ be an eigenvalue of $P$ and $f:=\{f(n)\}_{n\in\N}$ be a non trivial associated eigenvector. Then  
\begin{equation}
\lambda\, f(0) = \sum_{n=0}^{+\infty}\pi_n f(n) \qquad \text{and} \qquad
\forall n\geq1,\quad \lambda\, f(n) = \pi_0f(0)	+ (1-\pi_0)f(n). \label{fct-pr-2-bis}
\end{equation}
This gives: $\forall n\geq1,\quad f(n) = f(0) \pi_0/(\lambda-1+\pi_0)$. 
Since $f\neq0$, it follows from the first equality in (\ref{fct-pr-2-bis}) that  
$$\lambda = \pi_0	+ \frac{\pi_0}{\lambda-1+\pi_0}(1-\pi_0) ,$$
which is equivalent to $\lambda^2-\lambda=0$. Thus, $\lambda=1$ or $0$. That $1$ is a simple eigenvalue is standard from the irreducibility of $P$. The result follows from Proposition~\ref{CNS-qc-Vgeo}.
\end{proof}

A specific instance of this model is considered in \cite[p.~68]{Ros96}. 
Let $\{w_n\}_{n\geq1}$ be a sequence of positive scalars such that $\sum_{n\geq1} w_n=1/2$. Then $P$ is given by 
$$\forall n\in\N,\quad P(n,n)=1/2\quad \text{and} \quad  \forall n\ge 1,\  P(0,n)=w_n,\ P(n,0)=1/2$$
which is reversible with respect to its invariant probability distribution $\pi$ defined by $\pi_0:=1/2$ and $\pi_n:=w_n$ for $n\geq 1$. It has been proved in  \cite[p.~68]{Ros96} that, for any  $X_0\sim \alpha\in\ell^2(1/\pi)$, there exists a constant $C_{\alpha,\pi}>0$ such that
\begin{equation} \label{ineg-Ros-96}
\|\alpha P^n - \pi\|_{TV} \leq C_{\alpha,\pi}\, \left(3/4\right)^n
\end{equation}
where $\|\cdot\|_{TV}$ is the total variation distance. Since we know that $\rho_2(P)\le \rho_V(P)$ from \cite{Bax05} and $\rho_V(P)\le 1/2$ from Proposition~\ref{pro-rhoV-ex}, the rate of convergence in  (\ref{ineg-Ros-96}) is improved.  

\paragraph*{Acknowledgment} 
The authors thank Denis Guibourg for stimulating discussions about this work.

\end{document}